\documentclass[10pt,a4paper]{article}
\usepackage{amsmath,amssymb,amscd}
\usepackage{latexsym}
\usepackage{amsthm}
\usepackage{stmaryrd}
\usepackage[all]{xy}
\usepackage{graphicx}

\newtheorem{Def}{Definition}[section]
\newtheorem{Thm}[Def]{Theorem}
\newtheorem{Lem}[Def]{Lemma}
\newtheorem{Prop}[Def]{Proposition}
\newtheorem{Cor}[Def]{Corollary}
\newtheorem{Rem}[Def]{Remark}

\newtheorem{Ex}[Def]{Example}

\numberwithin{equation}{section}

\begin{document}

\title{The 12-th roots of the discriminant of an elliptic curve and the torsion points}
\author{Kohei Fukuda, Sho Yoshikawa}
\date{}
\maketitle

\begin{abstract}
Given an elliptic curve over a field of characteristic different from 2,3,
its discriminant defines a $\mu_{12}$-torsor over the field.
In this paper,
we give an explicit description of this $\mu_{12}$-torsor in terms of the 3-torsion points and of the 4-torsion points on the given elliptic curve.
As an application, we generalize a result of Coates on the 12-th root of the discriminant of an elliptic curve.
\end{abstract}

\section{Introduction}
Let $E$ be an elliptic curve over a field $K$.
Then, the \textit{discriminant} $\Delta_E$ \textit{of} $E$ is defined to be a value in $K^{\times}/(K^{\times})^{12}$, 
which depends only on the isomorphism class of $E$.
 
To study $\Delta_E$, it suffices to consider $\Delta_E \mathrm{~mod~} (K^{\times})^n$ for $n=3,4$ separately, because $K^{\times}/(K^{\times})^{12}\simeq K^{\times}/(K^{\times})^3\times K^{\times}/(K^{\times})^4$.
In the rest of this introduction, we suppose that $n=3,4$ and that $K$ be a field of characteristic prime to $n$.
Then, we may focus on the $\mu_n$-torsor $\mu_n\sqrt[n]{\Delta_E}$ over $K$ consisting of $n$-th roots of $\Delta_E$  since $\Delta_E \mathrm{~mod~} (K^{\times})^n$ corresponds to the $\mu_n$-torsor  $\mu_n\sqrt[n]{\Delta_E}$ 
via the isomorphism $K^{\times}/(K^{\times})^n\cong\mathrm{H}^1(K,\mu_n)$ from Kummer theory. 

The goal of this paper is to give an explicit description of the $\mu_n$-torsor $\mu_n\sqrt[n]{\Delta_E}$ out of $E[n]$.
More precisely, we construct a  ${\bigwedge}^2 E[n]$-torsor $T_n(E[n])$ over $K$ and define a canonical isomorphism $w_n$ from $T_n(E[n])$ to the $\mu_n$-torsor $\mu_n\sqrt[n]{\Delta_E}$ over $K$ such that the following diagram is commutative:
\[
\begin{CD}
{\bigwedge}^2 E[n]\times T_n(E[n]) @ >  \textrm{action} >> T_n(E[n]) \\
@V e_n\times w_n  VV  @VV w_n V\\
\mu_n\times \mu_n \sqrt[n]{\Delta_E} @ > \textrm{action} >> \mu_n \sqrt[n]{\Delta_E},
\end{CD}
\]
where $e_n$ is the Weil pairing normalized as in Remark \ref{RemNor} and Remark \ref{TateE}.

We now describe the organization of this paper.

In Section 2, we give a brief review of several results on elliptic curves.
The results on Tate curves and modular curves will be crucial in the proof of our main theorem.

Sections 3 and 4 give a preparation for stating our main theorem.
Our first task is to construct a ${\bigwedge}^2 E[n]$-torsor $T_n(E[n])$,
which are explained in Section 3 as we mentioned above.
Since this construction only uses the fact $E[n]\cong \left( \mathbb{Z}/(n)\right)^2$,
Section 3 deals with an abstract free $\mathbb{Z}/(n)$-module of rank $2$.
The second task is to construct a bijection $w_n:T_n(E[n])\rightarrow \mu_n\sqrt[n]{\Delta_E}$,
which is done in Section 4. Also, we give a simple description of the action of $\bigwedge^2 E[n]$ on $T_n(E[n])$ when $E$ is a Tate curve.
The constructions of $w_n$ are based on [4] and [7].

Next, integrating the results of the previous two sections, we proceed to the main part of this paper:
in Section 5, we give the precise statement of the main theorem
and prove it.
In its proof, using the modular curves of level 3 and 4,  
we reduce a claim on general elliptic curves to a claim on a single Tate curve.
As such an elliptic curve, we choose the Tate curve over a Laurent series field $\mathbb{Q}((q))$.
This choice enables us to compute concretely both the map $w_n:T_n(E[n])\rightarrow \mu_n\sqrt[n]{\Delta_E}$ and the Weil pairing for the Tate curve.

In the last section, we give a consequence of our main theorem extending the result of Coates [2] in the following sense:  \vspace{1em}\\
\textbf{Corollary \ref{Coa}.}
 \textit{Let $E$ and $E'$ be elliptic curves over $K$ of characteristic $\mathrm{char}(K)\neq 2,3$,
and $\varphi :E\rightarrow E'$ be an isogeny over $K$.
If $d=\mathrm{deg} \varphi$ is prime to $12$, then we have $\Delta_E = (\Delta_{E'})^d$ in $K^{\times}/(K^{\times})^{12}$.}

\section*{Acknowledgment}
The authors wish to express their hearty gratitude to their advisor, Professor Takeshi Saito.
He not only gave them  a lot of suggestive advice,
but also patiently encouraged them when they were in trouble with their study.
Also, the authors sincerely appreciate Professor J.-P. Serre,
because originally he suggested the theme of this paper to Professor Saito.
This work was supported by the Program for Leading Graduate 
 Schools, MEXT, Japan.

\section*{Convention and Terminology}
Let $\varphi:G\rightarrow H$ be a homomorphism of groups and $f:X\rightarrow Y$ be a map, where $X$ (resp. $Y$) is  a $G$-set (resp. $H$-set).
We say that $f$ \textit{is compatible with} $\varphi$ if for any $x\in X$ and $g\in G$ we have
$f(g\cdot x)=\varphi(g)\cdot f(x)$.\\
Also, we often denote by $\omega$ (resp. $i$) a primitive cubic (resp. 4th) root of unity.

\section{Review on elliptic curves}

In this section, we review some results on elliptic curves we need in this paper.
We omit all the proofs of these results, and only give references. 
\subsection{Elliptic curves}
\begin{Def}
\label{Ell}
Let $S$ be a scheme.
An elliptic curve over S is a pair $(E,O)$ 
such that $E$ is a proper smooth scheme over $S$
whose geometric fibers are connected algebraic curves of genus $1$,
and that $O$ is a section of $E\rightarrow S$.
\end{Def}
If there is no risk of confusion, we omit to write $O$ and simply call $E$ an elliptic curve.
It is well-known as the Abel-Jacobi theorem (THEOREM 2.1.2, [3]) that every elliptic curve $(E,O)$ admits a structure of
a commutative group scheme with $O$ the unit section.
For each positive integer $N$,
we denote by $E[N]$ the $N$-torsion subgroup scheme of $E$.
\begin{Rem}
\label{FEt}
If $N$ is invertible on $S$,
then $E[N]$ is a finite $\acute{e}tale$ commutative group scheme.
If, in addition, $S$ is a spectrum of a field $K$, 
then we often identify the finite $\acute{e}tale$ commutative group scheme $E[N]$ over $K$
with the associated $G_K=\mathrm{Gal}(\bar{K}/K)$-module $E[N](\bar{K})$.
\end{Rem}
\begin{Def}
\label{D}
For a Weierstrass equation over a ring $A$ of the form
\[ y^2 + a_1xy + a_3y = x^3 + a_2x^2 + a_4x + a_6, \]
we define $b_i\in A \, (i=2,4,6,8)$ by
\begin{equation*}
 \begin{split}
 b_2 &= a_1^2+4a_2, \\
 b_4 &= a_3^2+4a_6, \\
 b_6 &= 2a_4+a_1a_3, \\
 b_8 &= a_1^2a_6+4a_2a_6-a_1a_3a_4+a_2a_3^2+a_4^2,
 \end{split}
\end{equation*}
and define the discriminant $\Delta$ of the Weierstrass equation by
\[ \Delta := 9b_2b_4b_6-b_2^2b_8-8b_4^3-27b_6^2. \].
\end{Def}
The following theorem actually characterizes the definition of elliptic curves.
\begin{Thm}
\label{EllCh}
Let $E/S$ be an elliptic curve.
Then there exists an affine covering $\{ U_i\}_{i\in I}$ of $S$ such that, for each $i\in I$, 
$E_{U_i}\, /U_i$ is given by a Weierstrass equation over $\Gamma (U_i, \mathcal{O})$ 
with the discriminant invertible on $\Gamma (U_i, \mathcal{O})$.
\end{Thm}
\begin{proof}
See (2,2), [3]
\end{proof}
\begin{Rem}
\label{VarCh}
{\rm (1)
It is well-known (for example,  see [8]) that, for a Weierstrass equation, the only change of variables fixing $[0:1:0]$ and preserving the general Weierstrass form is
\begin{equation}
 \begin{split}
 x &= u^2X+r, \\
 y &= u^3Y+sX+t,
 \end{split}
\end{equation}
with $u\in A^\times$ and $r,s,t\in A$.
Let $b_i$ and $\Delta$ ( $b_i'$ and $\Delta'$, respectively) be the constants defined in Definition \ref{Ell} for the Weierstrass equation with coordinates $(x,y)$ ($(X,Y)$, respectively).
Then the above change of variable gives the following formula:
\begin{equation}
 \begin{split}
 u^2b_2' &= b_2+12r,\\
 u^4b_4' &= b_4+rb_2+6r^2, \\
 u^{12}\Delta' &= \Delta.
 \end{split}
\end{equation}
(2)
Suppose that 2 is invertible on a ring $A$.
Then, any Weierstrass equation can be made into the form
\[ y^2 = x^3 + a_2x^2 + a_4x + a_6 \]
over $A$, and  
\[ x = u^2X+r,\quad y = u^3Y\qquad (u\in A^\times, r\in A)\]
is the only change of variables fixing the point $[0:1:0]$ and preserving such a Weierstrass form.}
\end{Rem}
Let $E$ be an elliptic curve over a field $K$.
Take a Weierstrass equation for $E$ and consider its discriminant $\Delta$.
The Remark \ref{VarCh} (1) implies that the image of $\Delta$ in $K^{\times}/(K^{\times})^{12}$ only depends on $E$.
We denote it by $\Delta_E\in K^{\times}/(K^{\times})^{12}$,
and call it \textit{the discriminant of} $E$.
\subsection{The modular curve of level $r$ $\geq 3$}

In the following, we recall the representability of a moduli problem on elliptic curves.
This result will be used in the proof of our main theorem to reduce Proposition \ref{MT} to the Tate curve case (Lemma \ref{mu}). 
\begin{Thm}
\label{modular}
For a positive integer $r$,
let $\mathcal{M}(r): \mathit{Sch}/\mathbb{Z}[1/r]\rightarrow \mathit{Sets}$ be a functor defined as follows;
for a scheme $S$ over $\mathbb{Z}[1/r]$, $\mathcal{M}(r)(S)$ is the set of isomorphism classes of the pair $(E,\alpha )$,
where $E$ is an elliptic curve over S and $\alpha : \bigl( \mathbb{Z}/(r)\bigr)^2 \xrightarrow{~\cong~} E[r]$ is an isomorphism of group schemes over S.

(1) If $r=1$, then the morphism $j:\mathcal{M}\rightarrow \mathbb{A_Z}^1$ taking the $j$-invariant makes $\mathbb{A_Z}^1$ into the coarse moduli of $\mathcal{M}:=\mathcal{M}(1)$.
 
(2) If $r\geq 3$, then $\mathcal{M}(r)$ has the fine moduli  
$Y(r)$, which is a connected smooth curve over $\mathbb{Z}[1/r]$.
\end{Thm}
\begin{proof}
 See Corollary 8.40, [5] for (1) and Lemma 8.37, [5] for (2).
\end{proof}
In this paper, we call $Y(r)$ \textit{the modular curve of level} $r$.
\subsection{Weil pairing}

Next, we recall the Weil pairing of an elliptic curve.
We make explicit the sign convention of the pairing, which is an important point in our result.

\begin{Thm}
\label{WP}
Let $N$ be a positive integer.
For any elliptic curve $E/S$,
there exists a canonical bilinear pairing
\[ e_N:E[N]\times E[N]\longrightarrow \mu_N , \]
which is alternating and induces a self-duality of $E[N]$.
The construction of this pairing is functorial in the sense that
it defines a morphism of functors
\begin{equation*}
 \begin{split}
  \mathcal{M}(N) &\longrightarrow \mu_N . \\
  [(E/S,(P,Q))]  &\rightsquigarrow e_N(P,Q)
 \end{split}
\end{equation*}
\end{Thm}
\begin{proof}
See (2.8), [3].
\end{proof}
\begin{Rem}
\label{RemNor}
{\rm There are two choices of the sign of $e_N$.
We choose $e_N$ so that it satisfies the following equality:
for any elliptic curve $E\cong \mathbb{C}/(\mathbb{Z}+\mathbb{Z}\tau) ~(\tau \in \mathcal{H})$ over
$\mathbb{C}$, we have
\[ e_N\left( \frac{1}{N},\frac{\tau}{N}\right) = \exp\left( \frac{2\pi i}{N} \right). \]
We call the pairing of such a choice the normalized Weil pairing.}
\end{Rem}

\subsection{The Tate curve}
Finally, we recall some properties of the Tate curves.
The main references on this topic are [6] and [8].

Let $(K,v)$ be a complete discrete valuation field.
We fix an element $q\in K^\times$ satisfying $v(q)>0$.
Consider the elliptic curve (called the Tate curve) over $K$ defined by the following Weierstrass equation:
\[ E_q:  y^2+xy = x^3+\left( -5\sum_{n\geq 1}\frac{n^3q^n}{1-q^n} \right) x+\left( -\frac{1}{12}\sum_{n\geq 1}\frac{(7n^5+5n^3)q^n}{1-q^n} \right). \]
The discriminant $\Delta$ of this Weierstrass equation is given by
\begin{equation}
\label{TD}
 \Delta = q\prod_{m\geq1}(1-q^m)^{24}.
\end{equation}
The curve $E_q$ has the following important properties:
\begin{Thm}
\label{FactTate}
Let the notation be as above.
Let $u$ be an indeterminate,
and $x=x(u,q)$ and $y=y(u,q)$ be the two power series in $\mathbb{Z}[u,u^{-1}][[q]]$ defined by
\begin{equation}
\begin{split}
 x(u,q) &:= \sum_{n\in \mathbb{Z}}\frac{q^nu}{(1-q^nu)^2}-2\sum_{n\geq 1}\frac{nq^n}{1-q^n}
\end{split}
\end{equation}
\begin{equation}
\begin{split}
 y(u,q) &:= \sum_{n\in \mathbb{Z}}\frac{q^{2n}u^2}{(1-q^nu)^3}+\sum_{n\geq 1}\frac{nq^n}{1-q^n}.
\end{split}
\end{equation}
(1) Then the map
\begin{equation*}
 \begin{split}
{K^{\times}}/q^{\mathbb{Z}} &\longrightarrow E_q(K) \\
  u ~~ &\longmapsto (x(u,q),y(u,q))
 \end{split}
\end{equation*}
is an isomorphism of abelian groups.

More generally, the map
\begin{equation*}
 \begin{split}
{\bar{K}^{\times}}/q^{\mathbb{Z}} &\longrightarrow E_q(\bar{K}) \\
  u ~~ &\longmapsto (x(u,q),y(u,q))
 \end{split}
\end{equation*}
makes sense and is an isomorphism of $G_K$-modules.\\
(2) For each positive integer $N$ prime to the characteristic of $K$, there exists a canonical exact sequence
\begin{equation}
\label{TateEx}
 0\rightarrow \mu_N(\bar{K})\rightarrow E_q[N](\bar{K})\rightarrow \mathbb{Z}/(N)\rightarrow 0
\end{equation}
of $G_K$-modules.
\end{Thm}
\begin{Rem}
\label{TateE}
{\rm
Let us consider here the Tate curve $E_q$ over $K=\mathbb{Q}((q))$.
Then, the exact sequence (2.5) does not split as $G_K$-modules.
However, consider the base change of $E_q/K$ by the injection
\begin{equation*}
 \begin{split}
K &\longrightarrow L:=\mathbb{Q}(\zeta_N )((z)) \\
  q ~ &\longmapsto ~~ z^N.
 \end{split}
\end{equation*}
Then, we have the isomorphism of $G_L$-modules
\begin{equation*}
 \begin{split}
\mathbb{Z}/(N) &\xrightarrow{ ~~\cong~ } \mu_{N,L}(\bar{L}) \\
  1 ~~ &\longmapsto ~~ \zeta_N.
 \end{split}
\end{equation*}
We also obtain the canonical section
\begin{equation*}
 \begin{split}
\mathbb{Z}/(N) &\longrightarrow E_{q,L}[N](\bar{L}). \\
  1 ~~ &\longmapsto ~~z
 \end{split}
\end{equation*}
of the homomorphism $E_{q,L}[N](\bar{L})\rightarrow \mathbb{Z}/(N)$ in the exact sequence (2.5). 
Therefore, the exact sequence splits and gives a canonical identification
\begin{equation*}
 \begin{split}
 E_{q,L}[N] &\cong \mu_{N,L} \times \mathbb{Z}/(N)\\
        &\cong \mathbb{Z}/(N)\times \mathbb{Z}/(N).
 \end{split}
\end{equation*}
Under this identification, it is known that the normalized Weil pairing $e_N$ sends each
$((a,b),(c,d))\in \bigl( \mathbb{Z}/(N)\times \mathbb{Z}/(N)\bigr)^2$ to ${\zeta_N}^{ad-bc}$;
that is, $e_N(\zeta_N, z)=\zeta_N$.~(See [1, VII, 1, 16].)}
\end{Rem}

\section{Construction of torsors}
Our description of the 12-th roots of the discriminant of an elliptic curve by the torsion points consists of the following two parts:
one needs some geometric properties of elliptic curves, and the other only needs a bit knowledge of linear algebra.
In this section, we explain the linear-algebraic part.
Throughout this section,
we assume that $n=3,4$, and let $V$ denote a free $\mathbb{Z}/(n)$-module of rank $2$.
Our tasks here are to construct a set $T_n(V)$, which will later describe $\mu_n \sqrt[n]{\Delta_E}$ when $V=E[n]$,
and to define a simply transitive action of ${\bigwedge}^2V$ on $T_n(V)$.

\subsection{The construction of $T_3(V)$}
Here, $V$ is a $2$-dimensional vector space over $\mathbb{F}_3$.
We construct a set $T_3(V)$ attached to $V$.
We write $\mathbb{P}(V):= (V\setminus \{ 0 \})/\{ \pm 1 \}$ for the projective line associated with $V$, and $\overline{P}$ for the image of a point $P\in V\setminus \{ 0\}$ in $\mathbb{P}(V)$.
Note that $\# \mathbb{P}(V)=4$.
\begin{Def}
\label{tors3}
We define a set $T_3(V)$ by
\[ T_3(V) = \{\{X,Y\}| ~  X\sqcup Y=\mathbb{P}(V), \# X=\# Y=2 \}. \]
That is, this is the set of 2-2 partitions of $\mathbb{P}(V)$.
\end{Def}
For a basis $(P,Q)$ of $V$, we write $[P,Q]\in T_3(V)$ for $\{\{\overline{P},\overline{Q}\}, \{\overline{P+Q},\overline{P-Q}\}\}$.
Since $\mathbb{P}(V)=\{ \overline{P}, \overline{Q}, \overline{P+Q}, \overline{P-Q} \}$, the set $T_3(V)$ consists of 3 elements $[P,Q]$, $[P,P+Q]$, and $[P,P-Q]$.
\begin{Rem}
\label{simptr3}
{\rm
For the later use, note that $T_3(V)$ has canonically a simply transitive action of the alternating group $\mathcal{A}(\mathrm{Aut}(T_3(V)))$.}
\end{Rem}

\subsection{The construction of $T_4(V)$}
In this subsection, we define a set $T_4(V)$ for a  free $\mathbb{Z}/(4)$-module $V$ of rank $2$.
Our construction of $T_4(V)$ for $n=4$ is motivated by the definition of $w_4$ in Subsection 4.2.

Denote $V[2]=\ker (2\times :V\rightarrow V)$ and 
define a set $S_4(V)$  by
\[ S_4(V)=\bigl\{ (P,Q,R)\in V^{3}\, |\, \{2P,2Q,2R\}=V[2]\setminus 0 \bigr\}. \]
For each
$\sigma \in\mathfrak{S}_3,$
define an involution $[\sigma]$ of $S_4(V)$ by $[\sigma](P_1,P_2,P_3):=(P_1',P_2',P_3')$ with
\[ P_{\sigma(1)}'=P_{\sigma(1)}, P_{\sigma(2)}'=P_{\sigma(2)}+2P_{\sigma(1)}, \, \mathrm{and}\, P_{\sigma(3)}'=P_{\sigma(3)}.\]
Since $[\sigma]$'s commute with each other,
we obtain an action of $\mathbb{F}_2^{\mathfrak{S}_{\,3}}$ on $S_4(V)$.
Combining this action with a canonical action of $\mathfrak{S}_3$ on $S_4(V)$ by permutations, 
we obtain an action of $G:=\mathfrak{S}_3 \ltimes \mathbb{F}_2^{\mathfrak{S}_{\,3}}$ on  $S_4(V)$.
Here, the left action of $\mathfrak{S}_3$ on $\mathbb{F}_2^{\mathfrak{S}_{\,3}}$ defining the semidirect product is given by $\tau\cdot [\sigma] =[\tau\sigma]$ for $\sigma, \tau \in \mathfrak{S}_3$.
Let $N$ be the kernel of the composite
\[ r : \mathbb{F}_2^{\mathfrak{S}_{\,3}}\rightarrow \mathbb{F}_2^{\mathfrak{A}_{\,3}}\rightarrow \{\pm 1\}, \]
where the first map is the restriction to $\mathfrak{A}_3$ and the second map is the one sending $\sum a_\sigma [\sigma]$ to $(-1)^{\sum a_\sigma}$.
Since the action of $\mathfrak{S}_3$ on $\mathbb{F}_2^{\mathfrak{S}_{\,3}}$ defining the semidirect product induces an action of $\mathfrak{A}_3$ on $N$,
we obtain $H :=\mathfrak{A}_3\ltimes N\triangleleft G$.

\begin{Def}
\label{tors4}
Under the above setting,
define the set $T_4(V)$ by the quotient
\[T_4(V):= H\setminus S_4(V).\]
The equivalence class of $(P,Q,R)\in S_4(V)$ is denoted by $[P,Q,R]\in T_4(V)$.
\end{Def}
Note that $T_4(V)$ has an action of $G/H=\mathfrak{S}_3/\mathfrak{A}_3\times \{ \pm 1\}$.
\begin{Rem}
\label{mis4}
{\rm
(1)
The element $-1\in G/H$ acts on $T_4(V)$ by
\[ [P,Q,R]\mapsto -[P,Q,R]:=[-P,-Q,-R]. \]
Indeed, $-1=r(\sum_{\sigma\in \mathfrak{S}_3}[\sigma])$ and 
\begin{equation*}
\begin{split}
\sum_{\sigma\in \mathfrak{S}_3}[\sigma](P,Q,R)&=(P+2Q+2R,Q+2R+2P,R+2P+2Q)\\
                                       &=(-P,-Q,-R)
\end{split}
\end{equation*}
in $S_4(V)$.\\
(2)
Define a set $S_2(V[2])$ by
\[ S_2(V[2]) =\bigl\{ (A,B,C)\in V[2]^3 \, | \{ A,B,C \} = V[2]\setminus 0 \bigr\} \]
equipped with the canonical action by $\mathfrak{S}_3$.
Consider the quotient
\[ T_2(V[2]):= \mathfrak{A}_3\setminus S_2(V) \]
by the cyclic permutations;
that is, it is the set of the cyclic orders on $V[2]\setminus 0$.
The equivalence class of an element $(A,B,C)\in S_2(V[2])$ is denoted by $[A,B,C]\in T_2(V[2])$.\\
The map
\begin{equation*}
 \begin{split}
 \tilde{\pi} : S_4(V)~&\longrightarrow S_2(V[2])\\
  (P,Q,R)&\mapsto (2P,2Q,2R)
 \end{split}
\end{equation*}
is compatible with $\mathrm{pr}_1:G\rightarrow \mathfrak{S}_3$.
Hence, it induces a map
\[ \pi : T_4(V)\longrightarrow T_2(V[2]), \]
which is compatible with $\mathrm{pr}_1:G/H\rightarrow \mathfrak{S}_3/\mathfrak{A}_3$.
Thus $\mathfrak{S}_3/\mathfrak{A}_3\subset G/H$ switches the fibers of $\pi$ and $\{\pm 1\}=\ker (\mathrm{pr_1})\subset G/H$ preserves each fiber of $\pi$.}
\end{Rem}
\begin{Lem}
\label{SimpTr}
The action of $G$ on $S_4(V)$ is simply transitive.
\end{Lem}
\begin{proof}
We use the notation in Remark \ref{mis4}.
Since the action of $\mathfrak{S}_3$ on $S_2(V)$ is simply transitive and $\tilde{\pi}$ is compatible with $\mathrm{pr}_1: G\rightarrow \mathfrak{S}_3$, it is enough to prove that $\mathbb{F}_2^{\mathfrak{S}_{\,3}}=\ker (\mathrm{pr}_1)$ acts on each fiber of $\tilde{\pi}$ simply transitively.

To prove this, let us fix an element $(A_1,A_2,A_3)\in S_4(V)$ and consider its fiber $F$ of $\tilde{\pi}$.
We make an identification $f: \mathbb{F}_2^{\mathfrak{S}_{\,3}}\overset{\simeq}{\longrightarrow} V[2]^3$ by sending $[\sigma]$ to $(A'_1,A'_2,A'_3)$ with $A'_{\sigma(2)} = A_{\sigma(1)}$ and $A'_{\sigma(1)}=A'_{\sigma(3)}=0$.
If we define an action of $V[2]^3$ on $F$ by the componentwise addition,
then the action $\mathbb{F}_2^{\mathfrak{S}_{\,3}}$ on $F$ is identified via $f$ with the action $V[2]^3$ on it,
which obviously acts simply and transitively on $F$.
\end{proof}
\begin{Cor}
\label{ord}
We have $\# T_4(V)=4.$
\end{Cor}
\begin{proof}
By Lemma \ref{SimpTr}, the action of $G/H$ on $T_4(V)$ is simply transitive. Our assertion follows from this and $\# G/H=4$.
\end{proof}
\begin{Rem}
\label{act}
{\rm
Let $[P,Q,R]\in T_4(V)$.
Since $G/H=\mathfrak{A}_3\times \{\pm1\}$ is generated by $((12), 1)$ and $(\mathit{id}, -1)$,
we see that
\begin{equation*}
 \begin{split}
  T_4(V) &= (G/H) \cdot [P,Q,R]\\
       &=  \{ \pm[P,Q,R], \pm[Q,P,R]\}.
 \end{split}
\end{equation*}}
\end{Rem}

We use the following lemma and remark in the next subsection to define an action of $\bigwedge^2 V$ on $T_4(V)$.
\begin{Lem}
\label{group}
Let $X$ be a set consisting of 4 elements.
Denote by $S_1$ the set of cyclic subgroup of order 4 in $\mathrm{Aut}(X)$, and by $S_2$ the set of elements in $\mathrm{Aut}(X)$ of order 2 with no fixed points on $X$.

(1)
For any $C\in S_1$, the action of $C$ on $X$ is simply transitive.
In particular, the order 2 element $\tau$ in $C$ belongs to $S_2$.

(2)
Mapping each $C\in S_1$ to its element of order 2 defines a bijection $d: S_1\overset{\simeq}{\longrightarrow}S_2$.
For any $\tau\in S_2$, we denote $d^{-1}(\tau)$ by $C(\tau)$.
\end{Lem}
\begin{proof}
(1)
Considering the orbit decomposition of $X$ by $C$, we see that the action of $C$ on $X$ is transitive; if not, any element of $C$ must be of order $1$ or $2$, which is a contradiction. This shows (1) since $\# C= \# X (=4)$.

(2)
By (1), we can define the map $d$.
Note that $\mathrm{Aut}(X)$ acts transitively on both $S_1$ and $S_2$ by conjugation, and that $d$ commutes with these actions.
Also, $\#S_1=\#S_2=3$, since there are 6 elements of order 4 in $\mathrm{Aut}(X)$ and each $C\in S_1$ has two generators, and since $S_2$ is identified with the set of 2-2 partitions via the orbit decompositions.
These shows that $d$ is bijective.
\end{proof}

\subsection{The action of $\bigwedge^2 V$ on $T_n(V)$}

For each basis $(P,Q)$ of $V$, we denote by $\varphi_{P,Q}\in \mathrm{SL}(V)$ the map determined by $\varphi_{P,Q}(P)=P$ and $\varphi_{P,Q}(Q)=P+Q$.

\begin{Lem}
\label{hom}
There exists a unique surjective homomorphism $\varphi : {\bigwedge}^2V\twoheadrightarrow \mathrm{SL}(V)^{\mathrm{ab}}$ such that,
for any basis $(P,Q)$ of $V$, $\varphi$ maps $P\wedge Q$  to $\overline{\varphi_{P,Q}}\in \mathrm{SL}(V)^\mathrm{ab}$.
\end{Lem}
\begin{proof}
For a basis $(P,Q)$ of $V$, mapping a generator $P\wedge Q$ of $\bigwedge ^2V$ to $\overline{\varphi_{P,Q}}\in \mathrm{SL}(V)^\mathrm{ab}$ defines a homomorphism $\varphi : {\bigwedge}^2V\rightarrow \mathrm{SL}(V)^{\mathrm{ab}}$.
To see that $\varphi$ is independent of the choice of $(P,Q)$,
we claim that, for $u\in \mathrm{GL}(V)$,
we have $\varphi_{P,Q}^{\det u}\equiv \varphi_{u(P),u(Q)}$ in $\mathrm{SL}(V)^\mathrm{ab}$.
Let $v$ be an element of $\mathrm{GL}(V)$ given by $v(P)=P$ and $v(Q)=(\det u)Q$.
Then, we have $\varphi_{P,Q}^{\det u}=\varphi_{P,(\det u)Q}=v\varphi_{P,Q}v^{-1}$.
Also,  $\varphi_{u(P),u(Q)}=u\varphi_{P,Q}u^{-1}$.
Since $u^{-1}v\in \mathrm{SL}(V)$, the claim follows.

By the above claim, we see that $\varphi$ is surjective as follows.
Because $\mathrm{SL}(V)$ is generated by all the elements $\varphi_{u(P),u(Q)}=u\varphi_{P,Q}u^{-1}$ with $u\in \mathrm{GL}(V)$,
the claim shows that $\mathrm{SL}(V)^\mathrm{ab}$ is generated by $\overline{\varphi_{P,Q}}^{\pm 1}$, and thus by $\varphi(P\wedge Q)=\overline{\varphi_{P,Q}}$.
\end{proof}

We let $\mathrm{SL}(V)$ canonically act on $T_n(V)$ and consider the corresponding homomorphism 
 $\tilde{\psi}: \mathrm{SL}(V)\rightarrow \mathrm{Aut}(T_n(V))$.
\begin{Prop}
\label{SLab}
(1)
Let $C$ denote the subgroup $\mathfrak{A}(\mathrm{Aut}(T_3(V)))$ (resp. $C(-1)$ as in Lemma \ref{group} (2)) of $\mathrm{Aut}(T_n(V))$ for $n=3$ (resp. $n=4$).
Then, the map $\tilde{\psi}$ induces an isomorphism
\[ \psi: \mathrm{SL}(V)^{\mathrm{ab}} \overset{\simeq}{\longrightarrow}C. \]
(2)
The surjective map $\varphi: \bigwedge^2V\twoheadrightarrow \mathrm{SL}(V)^\mathrm{ab}$ in Lemma \ref{hom} is an isomorphism.
\end{Prop}
\begin{proof}
(1) (2)
We first prove that the map $\tilde{\psi}$ induces a \textit{surjective} homomorphism
\[ \psi: \mathrm{SL}(V)^{\mathrm{ab}}\twoheadrightarrow C.\]
Since $\mathrm{SL}(V)$ is generated by the elements $\varphi_{X,Y}$ for all bases $(X,Y)$ of $V$,
it suffices to show that, for 
a basis $(P,Q)$ of $V$, $\tilde{\psi}$ maps $\varphi_{P,Q}$ to a generator of $C$.

For $n=3$, this follows because $T_3(V)=\{[P,Q],[P,P+Q],[P,2P+Q]\}$ and $\varphi_{P,Q}([P,iP+Q])= [P,(i+1)P+Q]$.

We next consider the case $n=4$.
Since $\varphi_{P,Q}^2[P,Q,P+Q]=[P,Q+2P,(P+Q)+2P]=([\mathrm{id}]+[23])[P,Q,P+Q]=-[P,Q,P+Q]$,
we see that $\tilde{\psi}(\varphi_{P,Q})$ is of order 4. Further, by Lemma \ref{group}, $\tilde{\psi}(\varphi_{P,Q}^2)$ is of order 2 with no fixed points,
which implies that $\varphi_{P,Q}^2[Q,P,P+Q]$ must be $-[Q,P,P+Q]$. Therefore, $\tilde{\psi}(\varphi_{P,Q}^2)=-1$ and $C(-1)$ is generated by $\varphi_{P,Q}$.

The surjective maps $\psi$ and $\varphi$ are actually bijective because $\# \bigwedge^2V=\# C (=4)$.
\end{proof}

We now define an action of $\bigwedge^2 V$ on $T_n(V)$ by the composition
\[ \psi\circ \varphi: \bigwedge^2 V\overset{\simeq}{\longrightarrow} \mathrm{SL}(V)^\mathrm{ab} \overset{\simeq}{\longrightarrow} C\subset \mathrm{Aut}(T_n(V)). \]

\begin{Cor}
\label{SimpTr}
The action of $\bigwedge^2 V$ on $T_n(V)$ is simply transitive.
\end{Cor}

\begin{proof}
This follows from Remark \ref{simptr3}, Lemma \ref{group} (1), and Proposition \ref{SLab}.
\end{proof}

Next, we see that the action $\bigwedge^2V\curvearrowright T_n(V)$ has a simple description when $V$ is moreover accompanied with an extension
\begin{equation}
\label{ext}
0\rightarrow L \rightarrow V \overset{p}{\longrightarrow} \mathbb{Z}/(n)\rightarrow 0
\end{equation}
of $\mathbb{Z}/(n)$.
We observe the following:

(1)
Denote $p^{-1}(1)$ by $T$.
Then, mapping each $P\in L$ to $P\wedge Q$ with $Q\in T$ defines an isomorphism
$\epsilon: L\overset{\simeq}{\longrightarrow} \bigwedge^2 V$, which is independent of the choice of $Q\in T$

(2)
Corresponding to the above extension, we let $L$ act on $V$ by an injective homomorphism $\tilde{\varphi}: L\hookrightarrow \mathrm{SL}(V)$
defined by $\tilde{\varphi}(P)(Q):= Q+p(Q)P$ for $P\in L$ and $Q \in V$.
Note that, if we define a subgroup $M$ of $\mathrm{SL}(V)$ by
\[M=\{ \sigma \in \mathrm{SL}(V): \sigma(P)=P \textrm{ for all } P\in L  \}, \]
then $\tilde{\varphi}$ induces an isomorphism $\tilde{\varphi}:L\overset{\simeq}{\longrightarrow} M$,
whose inverse $f:M\rightarrow L$ is obviously given by $f(\sigma) := \sigma(Q)-Q$ with any $Q\in T$.
We also remark that the action $L\simeq M\curvearrowright V$ preserves $T$, which coincides with the canonical action by translations;
in particular, this action of $L$ on $T$ is obviously simply transitive.

\begin{Ex}
\label{Ex}
{\rm
(1) If $V=E[n]$ for a Tate curve $E$, then we have a canonical extension (\ref{TateEx}) in Theorem \ref{FactTate}. \\
(2) If we fix a basis $(P,Q)$ of $V$, then we have an extension
\begin{equation}
\label{extbasis}
 0\rightarrow \langle P\rangle \rightarrow V \overset{p}{\longrightarrow}\mathbb{Z}/(n)\rightarrow 0,
\end{equation}
where $p:V\rightarrow \mathbb{Z}/(n)$ is given by $p(P)=0$ and $p(Q)=1$.}
\end{Ex}

\begin{Rem}
\label{commutative}
{\rm
(1)
Any extension (\ref{ext}) can becomes of the form (\ref{extbasis}) by taking a basis $(P,Q)$ with $P$ a generator of $L$ and $Q\in T$.

(2)
 For any extension (\ref{ext}),
the following diagram is commutative:
\[
\begin{CD}
L @ > \epsilon >> \bigwedge^2 V \\
@V \tilde{\varphi} V  V @ V \varphi V  V\\
M @ > \mathrm{quotient} >> \mathrm{SL}(V)^\mathrm{ab}.
\end{CD}
\]
To see this, we may assume that (\ref{ext}) is of the form (\ref{extbasis}) by (1).
Then, we obtain $\tilde{\varphi}(P)(Q)=\varphi_{P,Q}$ and $\epsilon(P)=P\wedge Q$ so that $\varphi(\epsilon (P))=\overline{\varphi_{P,Q}}\in \mathrm{SL}(V)^\mathrm{ab}$.}
\end{Rem}

For the given extension (\ref{ext}), we define a map $\tau: T\rightarrow T_n(V)$ by $\tau(Q)=[P,Q]$ (resp. $\tau(Q) =[Q, P, -(P+Q)]$) for $n=3$ (resp. $n=4$) with $P$ a generator of $L$.

\begin{Lem}
\label{tau}
(1)
The map $\tau$ is independent of the choice of $P$.

(2)
The map $\tau$ is compatible with $\epsilon: L\overset{\simeq}{\longrightarrow} \bigwedge^2V$. 

(3)
The map $\tau$ is bijective.
\end{Lem}
\begin{proof}
(1)
The other generator of $L$ is $-P$ for both cases $n=3$ and $n=4$. If $n=3$, then it defines the same point in $\mathbb{P}(V)$ as $P$. If $n=4$, then we see that  $[Q,-P,-(-P+Q)]=([123]+[id]+[13])[Q,P,-(P+Q)]=[Q,P,-(P+Q)]$ for $Q\in T$. These prove (1).

(2)
First, note that $\tau$ is tautologically compatible with $\tilde{\varphi}: L\rightarrow \mathrm{SL}(V)$.
By Proposition \ref{SLab} (1), it is compatible with the composite map $L\overset{\tilde{\varphi}}{\longrightarrow} \mathrm{SL}(V)\twoheadrightarrow \mathrm{SL}(V)^\mathrm{ab}$,
which coincides with the composite map $L\overset{\epsilon}{\longrightarrow} \bigwedge^2V \overset{\varphi}{\longrightarrow} \mathrm{SL}(V)^\mathrm{ab}$ by Remark \ref{commutative} (2).
Thus (2) follows, since the action of $\bigwedge^2V$ on $T_n(V)$ is defined via $\varphi: \bigwedge^2V \overset{\simeq}{\longrightarrow} \mathrm{SL}(V)^\mathrm{ab}$.

(3)
Because the case $n=3$ is obvious from the definition of $\tau$, we show the case $n=4$.
By (2), we have an action of $L$ on $T':=\tau(T)$, which is transitive because the action of $L$ on $T$ is transitive.
The order 2 element $A$ of $L$ acts on $T'$ as $-1$; in fact, by Remark \ref{commutative}, we may assume that the given extension is of the form (\ref{extbasis}), in which case $A=2P$ acts on $T'$ as $-1$ as shown in Proposition \ref{SLab}.
Since $-1$ is fixed point free on $T_4(V)$, the stabilizer of any element of $T'$ in $L$ has no elements of order 2; that is, it is trivial.
This shows that $\# T'= \# T_4(V)(=4)$, and hence $\tau$ is bijective.
\end{proof}

\begin{Rem}
We can also deduce Corollary \ref{SimpTr} from Lemma \ref{tau} as follows.
Fix an extension (\ref{ext}).
By Lemma \ref{tau}, we identify the action $\bigwedge^2 V \curvearrowright T_n(V)$ with the action $L\curvearrowright T$, which is obviously simply transitive, 
\end{Rem}

\section{A bijection $w_n: T_n(E[n])\rightarrow \mu_n\sqrt[n]{\Delta_E}$ }
Here, let $n$ be $3$ or $4$, and $E$ be an elliptic curve given by a Weierstrass equation over a field $K$ of characteristic prime to $n$.
We apply the construction in Section 3 to $V=E[n]$, and construct a bijection $w_n: T_n(E[n])\rightarrow \mu_n\sqrt[n]{\Delta_E}$.
\subsection{The case $n=3$}
The result in this section is based on 5.5, [7].
Fix a Weierstrass equation
\begin{equation}
 y^2+a_1xy+a_3y=x^3+a_2x^2+a_4x+a_6
\end{equation}
for $E$, and denote its discriminant by $\Delta$.
Let $b_i$ $(i=2,4,6,8)$ be the constants defined in Definition \ref{D}.
For a point $P$ on $E$, we denote by $(x_P,y_P)$ its coordinate in terms of this Weierstrass equation,
and by $\Delta$ the discriminant of it.
If $P$ is a point of order 3, then $P$ satisfies $x_P = x_{-P}= x_{2P}$.
Since
\[ x_{2P}= \frac{x_P^4-b_4x_P^2-2b_6x_P-b_8}{4x_P^3+b_2x_P^2+2b_4x_P+b_6} \]
(III, 2.3, [8].),
the $x$-coordinate of any point of order 3 is a solution of the equation
\[ x=\frac{x^4-b_4x^2-2b_6x-b_8}{4x^3+b_2x^2+2b_4x+b_6};\]
that is,
\begin{equation}
 x^4+\frac{b_2}{3}x^3+b_4x^2+b_6x+\frac{b_8}{3}=0. 
\end{equation}
For $\{ X,Y \}\in T_3(E[3])$, let $x_1$ and $x_2$ be the $x$-coordinates of the elements in $X$, and $x_3$ and $x_4$ be the $x$-coordinates of the elements in $Y$.
Using this, we define 
\[ w_3(\{ X,Y \}) := b_4-3(x_1x_2+x_3x_4),\]
which is obviously $G_K$-equivariant.
\begin{Lem}
\label{cubic}
Let the notation be as above.
Then,
\begin{equation}
T^3-\Delta = \prod (T-w_3(\{ X,Y\})),
\end{equation}
where the product is taken over $T_3(E[3])$.
In other words,
$w_3$ defines a bijection $w_3:T_3(E[3])\rightarrow \mu_3\sqrt[3]{\Delta_E}$.
\end{Lem}
\begin{proof}
We first show our assertion when $E$ is defined over a field $K$ and (4.1) is of the Deuring normal form
\[ y^2+\alpha xy + y =x^3\quad (\alpha \in \bar{K},\, \alpha^3\neq 27). \]
In this case, we have
\[ b_2=\alpha^2, b_4=\alpha, b_6=1, b_8=0, \mathrm{and } \, \Delta =\alpha^3-27.\]
Then, (4.2) is
\begin{equation}
X^4+\frac{1}{3}\alpha^2X^3+\alpha X^2 + X=0,
\end{equation}
and the solutions of (4.4) are just the $x$-coordinates $x_1,...,x_4$ of $E[3]$;
one of which is $0$, say $x_4=0$.
It follows that
\begin{equation}
X^3+\frac{1}{3}\alpha^2X^2+\alpha X+1=(X-x_1)(X-x_2)(X-x_3).
\end{equation}
Note that, under our present assumption, (4.3) becomes
\[ T^3-(\alpha^3 -27) = \prod_{1\leq i<j\leq 3} (T-\alpha +3x_ix_j). \]
Replacing $Y=T-\alpha$ reduces the problem to showing that
\begin{equation}
Y^3+3\alpha Y^2 +3\alpha^2Y+27 = (Y+3x_1x_2)(Y+3x_2x_3)(Y+3x_1x_3).
\end{equation}
This is true from the following computation:
Setting $X=3/Y$ in (4.5) and using $x_1x_2x_3=-1$,
we see that (4.5) becomes
\begin{equation*}
\begin{split}
\frac{1}{Y^3}(27+3\alpha^2 Y+3\alpha Y^2+Y^3) &= \frac{1}{Y^3}(3-x_1Y)(3-x_2Y)(3-x_3Y)\\
														   &= -\frac{x_1x_2x_3}{Y^3}(-3x_1^{-1}+Y)(-3x_2^{-1}+Y)(-3x_3^{-1}+Y)\\
														   &= \frac{1}{Y^3}(3x_2x_3+Y)(3x_1x_3+Y)(3x_1x_3+Y),
\end{split}
\end{equation*}
which is none other than (4.6).

Next, we prove the general case.
Suppose we are given two Weierstrass equations for $E$ over $K$ with coordinates $(x,y)$ and $(x',y')$.
We use the notation in Remark 2.5. (1).
By (2.1) and (2.3), we obtain
\begin{equation}
b_4-3(x_ix_j+x_kx_l)=u^4(b_4'-3(x_i'x_j'+x_k'x_l')).
\end{equation}
By (2.3) and (4.7), it follows that (4.3) for $(x,y)$ is equivalent to (4.3) for $(x',y')$.
Then, the assertion follows from the fact that any elliptic curve over $K$ has a Weierstrass equation of the Deuring normal form over $\bar{K}$.
(This requires $\mathrm{char}(K)\neq 3$. For the detail, see Proposition 1.3, Appendix A, [8].)
\end{proof}

Let us check how the map $w_3$ changes if we take another Weierstrass equation for $E$.
Let the notation be as in Remark \ref{VarCh} (1), and denote by $w_3'$ the map as above for another Weierstrass equation for $E$ with the coordinate $(x',y')$.
\begin{Lem}
\label{trans3}
We have $w_3=u^4w_3'$.
\end{Lem}
\begin{proof}
Let $(P,Q)$ be a basis of $E[3]$.
Then $x_P$, $x_Q$, $x_{P+Q}$, and $x_{P-Q}$ are all distinct, and they must satisfy the above equation.
Considering the coefficients of $x^3$ in (4.2), we obtain
\[ x_P+x_Q+x_{P+Q}+x_{P-Q} = -\frac{b_2}{3}. \]
Using this equality and the relations in Remark 2.5 (1),
we obtain $w_3([P,Q]) =u^4 w_3'([P,Q])$. 
\end{proof}
Lemma 4.2 and $\Delta=u^{12}\Delta'$ imply that the diagram
\[
\xymatrix{
  & & \mu_3\sqrt[3]{\Delta}  \\
T_3(E[3]) \ar[urr]^{w_3} \ar[drr]^{w_3'} & & \\
  & & \mu_3\sqrt[3]{\Delta'} \ar[uu]_{u^4\times}
  }
\]
is commutative.
Since the vertical map $u^4\times$ is an isomorphism of $\mu_3$-torsors over $k$,
we identify $w_3$ and $w_3'$,
and we consider them as a map from $T_3(E[3])$ to the $\mu_3$-torsor $\mu_3\sqrt[3]{\Delta_E}$.
We denote this map again by $w_3$.
\begin{Rem}
\label{general3}
{\rm
The construction of $T_3(E[3])$ and $w_3$ are easily generalized to the case that $E$ is an elliptic curve over a scheme $S$ with 3 invertible.
This is done by considering the locally constant \'etale sheaf $E[3]$ over  $S$ and by \'etale descent.}
\end{Rem}
\subsection{The case $n=4$}
Since $K$ is of characteristic prime to 2, we can take a Weierstrass equation for $E$ of the form
\[ y^2 = x^3+a_2x^2+a_4x+a_6. \]
In this subsection, we only consider  Weierstrass equations of this form.
In order to write a $4$-th root of $\Delta_E$ in terms of $E[4]$,
first we prove the following lemma. 
(See also [4, \S 11].)
\begin{Lem}
\label{lem4}
Let the notation be as above.
Let $A\in E$ be a point of order 2, and $P$ and $P'$ be points of $E$ satisfying $2P=2P'=A$ and $P \neq \pm P'$.
Then we have
\[ \left( \frac{y_P-y_{P'}}{x_P-x_{P'}} \right)^2 = x_{A} -x_{P-P'}. \]
\end{Lem}

\begin{proof}
Denote $B$ and $C$ for the points of order $2$ on $E$ other than $A$.
Then, note that
\[ x^3+a_2x^2+a_4x+a_6 = (x-x_A)(x-x_B)(x-x_C), \]
and in particular we have $x_A+x_B+x_C=-a_2$.
Since $P+P'+(-P-P')=O$,
the three points $P,P'$, and $-P-P'$ are colinear.
So the solutions for the equation
\[ \left( \left( \frac{y_P-y_{P'}}{x_P-x_{P'}} \right) (x-x_P) +y_P \right)^2 = x^3+a_2x^2+a_4x+a_6\]
are exactly $x_P,x_{P'}$, and $x_{-P-P'} (= x_{P+P'})$.
Considering the coefficients of $x^2$, we have
\begin{equation}
x_P+x_{P'}+x_{P+P'}=\left( \frac{y_P-y_{P'}}{x_P-x_{P'}} \right)^2 -a_2.
\end{equation}

On the other hand, we show
\begin{equation}
x_P+x_{P'}=2x_{A}.
\end{equation}
To prove (4.9), let $T$ be a $4$-division point of $E$ with $2T = A$.
Then the same argument as above shows that
\begin{equation}
2x_T+x_A = \left( \frac{y_T-y_{A}}{x_T-x_{A}} \right)^2 -a_2.
\end{equation}
Combining (4.10) with
\begin{equation*}
 \begin{split}
 (y_T-y_{A})^2 &= {y_T}^2\\
              &= (x_T-x_A)(x_T-x_B)(x_T-x_C) ,
 \end{split}
\end{equation*}
we have the following quadratic equation for $x_T$:
\[ {x_T}^2 -(-a_2+x_A-x_B-x_C)x_T-({x_A}^2+a_2x_A+x_Bx_C)=0.\]
Since the solutions for this equation are exactly $x_P$ and $x_{P'}$,
we obtain
\begin{equation*}
 \begin{split}
 x_P+x_{P'} &= -a_2+x_A- x_B-x_C \\
            &= 2x_A,
 \end{split}
\end{equation*}
and hence (4.9) follows.

Now, combining (4.8) and (4.9), we obtain
\begin{equation}
 x_A+(x_A+x_{P+P'})=\left( \frac{y_P-{y_P}'}{x_P-{x_P}'} \right)^2 -a_2. 
\end{equation}
Since 
\[ \{ A,P+P',P-P' \} = E[2]\setminus O, \]
we have 
\begin{equation}
 x_A+x_{P+P'}+x_{P-P'}=-a_2. 
\end{equation}
Then, (4.11) and (4.12) show the lemma.
\end{proof}
An immediate corollary of Lemma \ref{lem4} is the following:

\begin{Cor}
\label{4th}
For $(P,Q,R)\in S_4(E[4])$,
set 
\[ P' =P+2Q, Q'=Q+2R, R'=R+2P.\]
Then
\[ \tilde{w}_4(P,Q,R):= 2\frac{y_P-y_{P'}}{x_P-x_{P'}}\cdot \frac{y_Q-y_{Q'}}{x_Q-x_{Q'}}\cdot \frac{y_R-y_{R'}}{x_R-x_{R'}} \]is a 4-th root of $\Delta$,
and the map $\tilde{w}_4: S_4(E[4])\rightarrow \mu_4\sqrt[4]{\Delta}$ makes the following diagram commutative:
\[
\begin{CD}
S_4(E[4]) @ > \tilde{w}_4 >> \mu_4\sqrt[4]{\Delta} \\
  @ VV\tilde{\pi} V @ VV\textrm{squaring} V\\
S_2(E[2]) @ > \tilde{w}_2 >> \pm \sqrt{\Delta},
\end{CD}
\]
where $\tilde{w}_2: S_2(E[2])\rightarrow \pm \sqrt{\Delta}$ is a canonical bijection given by
\[ \tilde{w}_2(A,B,C):=4(x_A-x_B)(x_B-x_C)(x_C-x_A). \]
\end{Cor}

\begin{Rem}
\label{ch4}
{\rm
Let the notation be as in Corollary \ref{4th}.
Then, we see that
\begin{equation*}
\begin{split}
\tilde{w}_4(-P,Q,R) &=-\tilde{w}_4(P,Q,R)\\
\tilde{w}_4(P',Q,R) &=\tilde{w}_4(P,Q,R)
\end{split}
\end{equation*}
for $(P,Q,R)\in S_4(E[4])$.
We also have the analogous properties for $Q$ and $R$.}
\end{Rem}
Suppose that we are given two Weierstrass equations defining $E$ with coordinates $(x,y)$ and $(x',y')$.
Denote the above maps by $\tilde{w}_4$ and $\tilde{w}_4'$,
and their discriminants by $\Delta_4$ and $\Delta_4'$, respectively.
\begin{Lem}
\label{trans4}
We have $\tilde{w}_4=u^3\tilde{w}_4'$. 
\end{Lem}
\begin{proof}
This is obvious from Remark \ref{VarCh} (2).
\end{proof}
 Lemma \ref{trans4} and $u^{12}\Delta'=\Delta$ give the following commutative diagram:
\[
\xymatrix{
  & & \mu_4\sqrt[4]{\Delta}  \\
S_4(E[4]) \ar[urr]^{\tilde{w}_4} \ar[drr]^{\tilde{w}_4'} & & \\
  & & \mu_4\sqrt[4]{\Delta'} \ar[uu]_{u^3\times}
  }
\]
By the same reason explained at the end of the last subsection,
we identify $\tilde{w}_4$ and $\tilde{w}_4'$,
and we consider them as a map from $S_4(E[4])$ to the $\mu_4$-torsor $\mu_4\sqrt[4]{\Delta_E}$.
Denote this map simply by $\tilde{w}_4$.
\begin{Lem}
\label{factor4}
The map $\tilde{w}_4: S_4(E[4])\rightarrow \mu_4\sqrt[4]{\Delta_E}$ factors through the quotient $T_4(E[4])$. 
\end{Lem}
\begin{proof}
We use the notation in Subsection 3.3 and Corollary \ref{4th}.
It is obvious that $\tilde{w}_4$ is invariant under the action of $\mathfrak{A}_3$.
We claim that for each $(P,Q,R)\in S_4(E[4])$ we have
\begin{equation}
\tilde{w}_4([\sigma](P,Q,R))= \left \{
\begin{array}{l}
-\tilde{w}_4(P,Q,R) ~ (\textrm{if} \,
\sigma
\in \mathfrak{S}_3 \textrm{ is even})\vspace{0.2cm}\\
\tilde{w}_4(P,Q,R)  ~ (\textrm{if} \,
\sigma
\in \mathfrak{S}_3 \textrm{ is odd}).
\end{array}
\right.
\end{equation}
Since every $[\sigma]$ for even (resp. odd) $\sigma$ is conjugate to $[\mathrm{id}]$ (resp. to $[(12)]$) by an element in $\mathfrak{A}_3\subset \mathrm{Aut}(S(E[4]))$,
it is enough to check the claim for $[\mathrm{id}]$ and $[(12)]$.
In these cases, we see that
\begin{equation*}
\begin{split}
\tilde{w}_4([\mathrm{id}](P,Q,R)) &= \tilde{w}_4(P,-Q',R)\\
								&=-\tilde{w}_4(P,Q,R)\\
\tilde{w}_4([(12)](P,Q,R)) &= \tilde{w}_4(P',Q,R)\\
								&=\tilde{w}_4(P,Q,R)
\end{split}
\end{equation*}
by Remark \ref{ch4}.
\end{proof}
Lemma \ref{factor4} gives us a map 
\begin{equation*}
\begin{split}
w_4:T_4(E[4]) &\longrightarrow \mu_4 \sqrt[4]{\Delta_E}\\
[P,Q,R]&\mapsto \tilde{w}_4(P,Q,R)
\end{split}
\end{equation*}
induced from $\tilde{w}_4$, which is $G_K$-equivariant as in the case $n=3$.
\begin{Rem}
\label{4and2}
{\rm
The commutative diagram in Corollary \ref{4th} induces a commutative diagram
\[
\begin{CD}
T_4(E[4]) @ > w_4 >> \mu_4\sqrt[4]{\Delta} \\
  @ VV\pi V @ VV\textrm{squaring} V\\
T_2(E[2]) @ > w_2 >> \pm \sqrt{\Delta}
\end{CD}
\]
of $G_K$-sets.
Here, $w_2$ is the map induced from $\tilde{w}_2$.
By Remark \ref{ch4}, $w_4$ is compatible with $\mathrm{pr}_2: G/H\rightarrow \{ \pm 1\}$.}
\end{Rem}
\begin{Cor}
\label{bij4}
The map $w_4:T_4(E[4])\rightarrow \mu_4\sqrt[4]{\Delta_E}$ is bijective.
\end{Cor}
\begin{proof}
Fix an element $X \in T_4(E[4])$.
Since $w_4$ is compatible with $\mathrm{pr}_2:G/H\rightarrow \{\pm 1\}$,
\begin{equation*}
 \begin{split}
w_4(T_4(E[4])) &= w_4((G/H)\cdot X)\\
           &= \{ \pm w_4(X), \pm w(\sigma X) \}
 \end{split}
\end{equation*}
for any odd permutation $\sigma\in \mathfrak{S}_3$.
The commutativity of the diagram in Remark 4.11 shows that
$\pm w_4(X)$ and $\pm w(\sigma X)$ belong to distinct fibers of the squaring map
 $\mu_4\sqrt[4]{\Delta_4}\rightarrow \pm \sqrt{\Delta_4}$.
Thus, we obtain $\# w_4(T_4(E[4]))= 4$ and hence $w_4$ is bijective.
\end{proof}
\begin{Rem}
\label{general4}
{\rm
For the same reason as in Remark \ref{general3}, we can define $T_4(E[4])$ and $w_4$
 for an elliptic curve over a scheme on which 4 is invertible.}
\end{Rem}

\subsection{The Tate curve case}

Let $E=E_q$ be the Tate curve over the field $K=\mathbb{Q}((q))$ of Laurent series.
By Theorem \ref{FactTate} (1), $E[n]$ is canonically an extension of $\mathbb{Z}/(n)$ in the category of $G_K$-modules:
\[ 0\rightarrow \mu_n\rightarrow E[n] \rightarrow \mathbb{Z}/(n)\rightarrow 0. \]
For this extension, we use the notations and results in Subsection 3.3 by setting $V=E[n$].
In our Tate curve case, we identify
\begin{equation*}
 \begin{split}
 L &= \mu_n,\\
 T &= \mu_n\sqrt[n]{q},\, \textrm{and}\\
 \epsilon &= e_n^{-1}: L\overset{\simeq}{\longrightarrow} \bigwedge^2 E[n],
\end{split}
\end{equation*}
the last equality of which is due to Remark \ref{TateE}.

Note that the action of $G_K$ on $E[n]$ gives rise to an isomorphism
\[ g:\mathrm{Gal}(K(\mu_n, \sqrt[n]{q})/K(\mu_n))\overset{\simeq}{\longrightarrow} M. \] 
We also remark that the composite map $f\circ g:\mathrm{Gal}(K(\mu_n, \sqrt[n]{q})/K(\mu_n))\overset{\simeq}{\longrightarrow} M \overset{\simeq}{\longrightarrow} L =\mu_n$ is the Kummer character $h_1$ mapping $\sigma \in \mathrm{Gal}(K(\mu_n, \sqrt[n]{q})/K(\mu_n))$ to $\sigma (\sqrt[n]{q})/\sqrt[n]{q}$.

\begin{Prop}
\label{mu}
The composition $w_n\circ \tau: T\rightarrow \mu_n \sqrt[n]{\Delta_E}$ is an isomorphism of $L=\mu_n$-torsors over $K$.
\end{Prop}
\begin{proof}
It is obvious from the constructions that $w_n\circ \tau$ is $G_K$-equivariant.
We show that it is also $\mu_n$-equivariant.
By definition, the action of $L$ on $T$ is identified with the action of $M$ on $T$ via the isomorphism $\tilde{\varphi}:L\overset{\simeq}{\longrightarrow} M$,
which is also identified with the action of $\mathrm{Gal}(K(\mu_n, \sqrt[n]{q})/K(\mu_n))$ on $T$ via $g:\mathrm{Gal}(K(\mu_n, \sqrt[n]{q})/K(\mu_n))\overset{\simeq}{\longrightarrow} M$.
On the other hand, the action of $G_K$ on $\mu_n\sqrt[n]{\Delta}$ induces an action of $\mathrm{Gal}(K(\mu_n, \sqrt[n]{q})/K(\mu_n))$ on  $\mu_n\sqrt[n]{\Delta}$,
 which is identified with a canonical action of $\mu_n$ on $\mu_n\sqrt[n]{\Delta}$ via the Kummer character $h_2:\mathrm{Gal}(K(\mu_n, \sqrt[n]{q})/K(\mu_n))\overset{\simeq}{\longrightarrow} \mu_n$ mapping $\sigma$ to $\sigma (\sqrt[n]{\Delta})/\sqrt[n]{\Delta}$.
Since $w_n\circ \tau$ is $G_K$-equivariant, it is also $\mathrm{Gal}(K(\mu_n, \sqrt[n]{q})/K(\mu_n))$-equivariant.
These arguments imply that the map $w_n\circ \tau: T\rightarrow \mu_n \sqrt[n]{\Delta_E}$ is compatible with $h_2\circ (f\circ g)^{-1}=h_2\circ h_1^{-1}: L\rightarrow \mu_n$, which is the identity map since $\sqrt[n]{\Delta}/\sqrt[n]{q}\in K(\mu_n)^{\times}$ by (\ref{TD}).
Therefore, $w_n\circ \tau$ is $L=\mu_n$-equivariant.
\end{proof}

\begin{Cor}
\label{CorComp}
The map $w_n$ is compatible with the normalized Weil pairing $e_n$.
\end{Cor}

\begin{proof}
This immediately follows from Corollary \ref{tau} (2) (3), Proposition \ref{mu}, and $e_n=\epsilon^{-1}$.
\end{proof}

Next we consider the following canonical isomorphism 
\begin{equation*}
 \begin{split}
\delta :  T &\rightarrow \mu_n\sqrt[n]{\Delta_E}\\
   z & \mapsto  z\prod_{m\geq 1}(1-q^m)^{24/n}
 \end{split}
\end{equation*}
of $L=\mu_n$-torsors over $K$.

\begin{Prop}
\label{cru}
We have $w_n\circ \tau=\delta$.
\end{Prop}

\begin{proof}
By Proposition \ref{mu}, the composite map $\sigma:= (w_n\circ \tau)\circ \delta^{-1}$ is an automorphism of the $\mu_n$-torsor $\mu_n\sqrt[n]{\Delta_E}$ over $K$,
and hence $\sigma$ belongs to $\mu_n(K)\subset \mathrm{Aut}(\mu_n\sqrt[n]{\Delta_E})$.
Because $\mu_n(K)=\{ 1\}, \{\pm 1\}$ for $n=3, 4$ respectively, the assertion for $n=3$ is proved, but there remains a possibility of $w_4\circ \tau=-\delta$.

To determine the sign of $\sigma$ for $n=4$, 
we check that the first coefficient of $w_4(\tau(z))$, which belongs to $\mathbb{Q}(i)((z))$, coincides with that of $\delta(z)=z\prod_{m\geq 1}(1-q^m)^6$; that is, $z$.
Note that, if $(P,Q,R)=\tau(z)=(z,i,(iz)^{-1})$ and if $P',Q',R'$ are as in Corollary \ref{4th}, then we obtain $(P',Q',R')=(-z,-iz^2,-iz)$.

Change variables of $E$ as
\[ X=x \quad \mathrm{and} \quad Y=y+\frac{1}{2}x\]
to make the Weierstrass equation of $E$ into the form
\[ Y^2=X^3+A_2X^2+A_4X+A_6. \] 
Then, it follows from Theorem 2.9 that
\begin{equation*}
 \begin{split}
\bar{K}^{\times}\big{/}q^\mathbb{Z} &\stackrel{\simeq}{\longrightarrow}
  \left\{ \begin{tabular}{l}
    $\bar{K}$-valued points on\\
    $y^2+xy=x^3+a_4x+a_6$
  \end{tabular} \right\}
\stackrel{\simeq}{\longrightarrow} 
 \left\{ \begin{tabular}{l}
    $\bar{K}$-valued points on\\
    $Y^2=X^3+A_2X^2+A_4X+A_6$
  \end{tabular} \right\} \\
u ~~~~&\longmapsto ~~~~~~ (x(u,q), y(u,q)) ~~~~~~~\longmapsto ~~~(x(u,q),\, y(u,q)+\frac{1}{2}x(u,q))
 \end{split}
\end{equation*}
We have
\[ x(u) := x(u,z^4)= f(u) + \sum_{n\geq 1}\left( \frac{z^{4n}u}{(1-z^{4n}u)^2} + \frac{z^{4n}u^{-1}}{(1-z^{4n}u^{-1})^2}-2\frac{ n z^{4n}}{1-z^{4n}} \right),\]
where $f(u)=\frac{u}{(1-u)^2}\equiv u+2u^2$ $(\mathrm{mod}\, z^3)$.
Also, set
\begin{equation*}
\begin{split}
Y(u) &:= y(u,z^4)+\frac{1}{2}x(u,z^4)= \frac{1}{2}\sum_{n\in \mathbb{Z}} \frac{z^{4n}u(1+z^{4n}u)}{(1-z^{4n}u)^2}\\
        &= g(u) + \frac{1}{2}\sum_{n\geq 1}\left( \frac{z^{4n}u^{-1}(z^{4n}u^{-1}+1)}{(z^{4n}u^{-1}-1)^3}+
\frac{z^{4n}u(1+z^{4n}u)}{(1-z^{4n}u)^3}\right),
\end{split}
\end{equation*}
where $g(u)=u(1+u)/2(1-u)^3\equiv\frac{1}{2}u+2u^2$ (mod $z^3$).
Note that the second term of $x(u)$ and $Y(u)$ is congruent to $0 \, (\mathrm{mod}\, z^2)$ whenever the degree of $u\in \mathbb{Q}(i)((z))$ in $z$ is of $0,\pm 1,\pm2$, and that $f(u)=f(u^{-1})$ and $g(u)=-g(u^{-1})$.
Then, the computation goes as follows:
\begin{equation}
 \begin{split}
w_4([z,i,iz]) &= 2\frac{Y(z)-Y(-z)}{x(z)-x(-z)}\cdot\frac{Y(i)-Y(-iz^2)}{x(i)-x(-iz^2)}\cdot\frac{Y((iz)^{-1})-Y(-iz)}{x((iz)^{-1})-x(-iz)}\\
           &\equiv 2\frac{g(z)-g(-z)}{f(z)-f(-z)}\cdot\frac{g(i)-g(-iz^2)}{f(i)-f(-iz^2)}\cdot\frac{g((iz)^{-1})-g(-iz)}{f((iz)^{-1})-f(-iz)}\quad (\mathrm{mod}\, z^2).
 \end{split}
\end{equation}
For each factor in the right hand side of (4.14), we see that
\begin{equation*}
 \begin{split}
\frac{g(z)-g(-z)}{f(z)-f(-z)} &\equiv \frac{z}{2z}\quad (\mathrm{mod}\, z^2)\\
      &\equiv \frac{1}{2}\quad (\mathrm{mod}\, z^2),
 \end{split}
\end{equation*}
\begin{equation*}
 \begin{split}
\frac{g(i)-g(-iz^2)}{f(i)-f(-iz^2)} &\equiv \frac{g(i)}{f(i)} \quad (\mathrm{mod}\, z^2)\\
      &\equiv \frac{i}{2}\quad (\mathrm{mod}\, z^2),
 \end{split}
\end{equation*}
and
\begin{equation*}
 \begin{split}
\frac{g((iz)^{-1})-g(-iz)}{f((iz)^{-1})-f(-iz)} &= \frac{-g(iz)-g(-iz)}{f(iz)-f(-iz)}\\
      &\equiv \frac{2z}{i}\quad (\mathrm{mod}\, z^3).
 \end{split}
\end{equation*}
Therefore,
\begin{equation*}
 \begin{split}
w_4([z,i,(iz)^{-1}]) &\equiv 2\cdot\frac{1}{2}\cdot\frac{i}{2}\cdot\frac{2z}{i}\quad (\mathrm{mod}\, z^2)\\
				  &\equiv z\quad (\mathrm{mod}\, z^2).
 \end{split}
\end{equation*}
Therefore, we have $w_4([z,i,(iz)^{-1}])=z\prod_{m\geq 1}(1-q^m)^6$.
\end{proof}

\section{The main theorem}
Combining the results of the previous two sections,
we now state our main theorem.
In the following statement, we consider $T_n(E[n])$ as a $\mu_n$-torsor via the action of $\bigwedge^2V$ on $T_n(E[n])$ by the identification $e_n: \bigwedge^2V \overset{\simeq}{\longrightarrow} \mu_n $ with $e_n$  the normalized Weil pairing.
\begin{Thm}
\label{MT}
Let $n $ be $3$ or $4$.
 
(1) The family of maps
\[ \bigl(w_n: T_n(E[n]) \longrightarrow \mu_n\sqrt[n]{\Delta_E}\bigr)_{E/S} \]
 defined in Section 3 with $E$ an elliptic curve over a scheme $S$ on which $n$ is invertible satisfies the following properties:\\
(a) Each map $w_n$ is an isomorphism of $\mu_n$-torsors over $S$.\\
(b) The maps $w_n\textrm{'s}$ are compatible with arbitrary base change.\\
(c) If $E=E_q$ is the Tate curve over $K=\mathbb{Q}((q))$, then the map $w_n$ coincides with $\delta\circ \tau^{-1}$.

(2) If $n=3$, then there exists a unique family of maps $\bigl(T_3(E[n])\overset{\simeq}{\longrightarrow} \mu_3\sqrt[3]{\Delta_E}\bigr)_{E/S}$ satisfying (a) and (b) with $E$ an elliptic curve over a scheme $S$ with 3 invertible. Hence it coincides with $\bigl(w_3: T_3(E[3])\rightarrow \mu_3\sqrt[3]{\Delta_E}\bigr)_{E/S}$ and automatically satisfies the property (c).

(3) If $n=4$, then $(w_4)_{E/S}$ and $(-w_4)_{E/S}$ are the only families of maps $T_n(E[4]) \longrightarrow \mu_4\sqrt[4]{\Delta_E}$ satisfying (a) and (b) with $E$ an elliptic curve over a scheme $S$ with $4$ invertible. Also, $(w_4)_{E/S}$ is the only family which satisfies the properties (a), (b), and (c).
\end{Thm}
\begin{proof}
(1)
Since (b) is obvious from our construction of $w_n$ and (c) follows from Proposition \ref{cru}, we have only to show (a) for $w_n$.
To do this, we may assume that $E[n]$ (and hence $\mu_n$) is constant over $S$.
Let $c_n$ be the conjugation
\[\mathrm{Aut}(T_n(E[n]))\overset{\simeq}{\longrightarrow} \mathrm{Aut}(\mu_n\sqrt[n]{\Delta_E}) \]
by $w_n$.
We consider $\bigwedge^2 E[n]$ as a subgroup of $\mathrm{Aut}(T_n(E[n]))$ by (3.1) and (3.3).
We also consider $\mu_n$ as a subgroup of $\mathrm{Aut}(\mu_n\sqrt[n]{\Delta_E})$ by the canonical action of $\mu_n$ on $\mu_n\sqrt[n]{\Delta_E}$.
For the proof, we first show the following lemma:
\begin{Lem}
\label{factor}
(1) The restriction of  the map $c_n$ to $\bigwedge^2 E[n]$ induces an isomorphism $e_n':\bigwedge^2 E[n]\overset{\simeq}{\longrightarrow} \mu_n$.

(2) The isomorphism $e_n':\bigwedge^2 E[n]\overset{\simeq}{\longrightarrow} \mu_n$ is the unique map with which $w_n$ is compatible.
\end{Lem}
\begin{proof}
(1)
If $n=3$, then the claim follows from the fact that $\mathrm{Aut}(T_3(E[3]))\simeq \mathrm{Aut}(\mu_3\sqrt[3]{\Delta_E})\simeq \mathfrak{S}_3$ has the unique subgroup of order 3; the alternating group.
If $n=4$, by Proposition \ref{SLab}, we have $\bigwedge^2E[4] =C(-1)\subset \mathrm{Aut}(T_4(E[4]))$ and $\mu_4=C(-1)\subset \mathrm{Aut}(\mu_4\sqrt[4]{\Delta_E})$.
Thus, it suffices to check that $c_4$ maps $-1\in \mathrm{Aut}(T_4(E[4]))$ to $-1\in \mathrm{Aut}(\mu_4\sqrt[4]{\Delta_E})$;
that is, $w_4(-X)=-w_4(X)$ for $X\in T_4(E[4])$.
This follows from Remark \ref{ch4}.

(2)
The compatibility is obvious from the definition, and the uniqueness follows because $\mu_n\rightarrow \mathrm{Aut}(\mu_n\sqrt[n]{\Delta_E})$ is injective.
\end{proof}

We continue the proof of the theorem.
To show the claim (1), we check that $e_n=e_n'$.

Denote $A:=\Gamma (Y(n),\mathcal{O})$ and let $(E_0/A, (P_0, Q_0))$ be the universal object of $\mathcal{M}(n)$ (see Theorem \ref{modular}).
Then the point $e_n'(P_0\wedge Q_0)\in \mu_n(A)$ defines a morphism $Y(n)\rightarrow \mu_n= \mathrm{Spec} [X, 1/n](X^n-1)$ of schemes over $\mathbb{Z}[1/n]$, which we also denote by $e_n'$.
This morphism satisfies the following property:
If $K$ is a field and $x=(E/K, (P,Q))\in Y(n)(K)$, then $e_n'(x)=e_n'(P\wedge Q)$.
In the same way, the point $e_n(P_0\wedge Q_0)\in \mu_n(A)$ also defines a morphism $e_n:Y(n)\rightarrow \mu_n$ of schemes over $\mathbb{Z}[1/n]$,
which satisfies $e_n(x)=e_n(P\wedge Q)$ for any $x$ as above.

Since $Y(n)$ is connected by Theorem \ref{modular},
the morphism $e_n/e_n': Y(n)\rightarrow \mu_n$ factors through a connected component $U$ of the scheme $\mu_n$ over $\mathbb{Z}[1/n]$.
Corollary \ref{CorComp} and Lemma \ref{factor} (2) imply that $e_n/e_n'(E_{z^n},(\zeta_n, z))=1$ for the Tate curve $E_{z^n}/\mathbb{Q}(\zeta_n)((z))$.
Therefore, $U$ must be the connected component $\mu_1=\mathrm{Spec}[X,1/n]/(X-1)$ of $\mu_n$.
This completes the proof of (1).

(2), (3) Suppose that we are given two maps $W_i: T_n(E[n])\rightarrow \mu_n\sqrt[n]{\Delta_E}$ ($i=1, 2$) for each elliptic curve $E/S$, which satisfy the condition (a) and (b) in the theorem.
Then $W_1/ W_2$ defines a morphism $W: \mathcal{M}\rightarrow \mu_n$ of functors $\mathit{Sch}/\mathbb{Z}[1/n]\rightarrow \mathit{Sets}$, where $\mathcal{M}$ is the functor defined in Theorem \ref{modular} (1).
By Theorem \ref{modular} (1), there exists a unique morphism $W': \mathbb{A}_{\mathbb{Z}[1/n]}^1\rightarrow \mu_n$ of schemes over $\mathbb{Z}[1/n]$ satisfying $W'\circ j=W: \mathcal{M}\rightarrow \mu_n$, where $j: \mathcal{M}\rightarrow \mathbb{A}_{\mathbb{Z}[1/n]}^1$ is given by taking the $j$ invariant.
The morphism $W'$ corresponds to an element of $\mu_n(\mathbb{Z}[1/n,j])=1, \{ \pm 1\}$ for $n=3,4$, respectively.
This proves the assertions.
\end{proof}

\begin{Rem}
\label{12}
{\rm
When $\mathrm{char}(k)\nmid 12$, we define $T_{12}(E[12])= T_3(E[3])\times T_4(E[4])$.
Then, $T_{12}(E[12])$ admits an action of ${\bigwedge}^2 E[12]\cong {\bigwedge}^2 E[3]\times{\bigwedge}^2 E[4]$, and Theorem 5.1 for $n=3,4$ immediately gives the analogous result for $n=12$.}
\end{Rem}

\section{An isogeny of an elliptic curve and the 12-th roots of its discriminant}
The following corollary gives a variant of a Coates' result [2, appendix].
The original result assumed that the characteristic of the base field is $0$.
\begin{Cor}
\label{Coa}
Let $E$ and $E'$ be elliptic curves over a field $K$ of characteristic $\mathrm{char}(K)\neq 2,3$,
and $\varphi :E\rightarrow E'$ be an isogeny over $K$.
If $d=\mathrm{deg} \varphi$ is prime to $12$, then we have $\Delta_E = (\Delta_{E'})^d$ in $K^{\times}/(K^{\times})^{12}$.
\end{Cor}
\begin{proof}
By Theorem \ref{MT}, we have the following commutative diagram:
\[
\begin{CD}
\mu_{12}\times \mu_{12} \sqrt[12]{\Delta_{E}} @ >  >> \mu_{12} \sqrt[12]{\Delta_{E}} \\
  @A  e_{12}\times w_{12} AA    @ AA w_{12} A\\
{\bigwedge}^2 E[12]\times T(E[12]) @ >   >> T(E[12]) \\
  @V \wedge^2 \varphi \times T\varphi VV   @ VV T\varphi V\\
{\bigwedge}^2 E'[12]\times T(E'[12]) @ >   >> T(E'[12]) \\
  @V e_{12} \times w_{12}' VV   @ VV w_{12}' V\\
\mu_{12}\times \mu_{12} \sqrt[12]{\Delta_{E'}} @ >  >> \mu_{12} \sqrt[12]{\Delta_{E'}} \\
  @V (\cdot )^{12}\times  (\cdot )^{12} VV     @ VV  (\cdot )^{12} V\\
\mu_{12}\times \mu_{12} \sqrt[12]{\Delta_{E'}^d} @ >  >> \mu_{12} \sqrt[12]{\Delta_{E'}^{d}}.
\end{CD}
\]
Here, $(\cdot)^{12}$ denotes the 12-th power map, and
 the horizontal maps are the action map.
Also, the vertical maps in the above diagram are all bijective and $G_K$-equivariant.
Since $e_{12}\circ\wedge^2\varphi \circ e_{12}^{-1}=(\cdot )^d: \mu_{12}\rightarrow \mu_{12}$ (for example, see [8, III, Proposition 8.2])
and $d^2 \equiv 1 \, (\mathrm{mod}\, 12)$,
the outside rectangle in the above diagram becomes
\[
\begin{CD}
\mu_{12}\times \mu_{12} \sqrt[12]{\Delta_{E}} @ >  >> \mu_{12} \sqrt[12]{\Delta_{E}} \\
@V \rotatebox{90}{$\simeq$} V \mathrm{id}\times\sigma V @ V \rotatebox{90}{$\simeq$} V  (\cdot)^{12}\circ w_{12}'\circ T_{\varphi} w_{12}^{-1} V\\
\mu_{12}\times \mu_{12} \sqrt[12]{\Delta_{E'}^d} @ >  >> \mu_{12} \sqrt[12]{\Delta_{E'}^{d}}.
\end{CD}
\]
This implies that $\mu_{12}\sqrt[12]{\Delta_{E}}$ and $\mu_{12}\sqrt[12]{\Delta_{E'}^d}$ are isomorphic as $\mu_{12}$-torsors over $K$. 
Hence the assertion holds.
\end{proof}

\end{document}